\documentclass[11pt]{amsart}
\usepackage{xypic,amsmath,amssymb}

\newtheorem{proposition}{Proposition}[section]
\newtheorem{lemma}[proposition]{Lemma}
\newtheorem{corollary}[proposition]{Corollary}
\newtheorem{theorem}[proposition]{Theorem}

\theoremstyle{definition}
\newtheorem{definition}[proposition]{Definition}

\theoremstyle{remark}

\newcommand{\thlabel}[1]{\label{th:#1}}
\newcommand{\thref}[1]{Theorem~\ref{th:#1}}
\newcommand{\selabel}[1]{\label{se:#1}}
\newcommand{\seref}[1]{Section~\ref{se:#1}}
\newcommand{\lelabel}[1]{\label{le:#1}}
\newcommand{\leref}[1]{Lemma~\ref{le:#1}}
\newcommand{\prlabel}[1]{\label{pr:#1}}
\newcommand{\prref}[1]{Proposition~\ref{pr:#1}}
\newcommand{\colabel}[1]{\label{co:#1}}
\newcommand{\coref}[1]{Corollary~\ref{co:#1}}

\newcommand{\delabel}[1]{\label{de:#1}}

\newcommand{\eqlabel}[1]{\label{eq:#1}}
\newcommand{\equref}[1]{(\ref{eq:#1})}

\def\equal#1{\smash{\mathop{=}\limits^{#1}}}

\newcommand{\End}{{\rm End}}

\newcommand{\Ker}{{\rm Ker}}

\newcommand{\im}{{\rm Im}}

\def\lan{\langle}
\def\ran{\rangle}
\def\ot{\otimes}

\def\sq{\square}

\def\units{{\mathbb G}_m}
\def\GG{{\mathbb G}}

\newcommand{\Cc}{\mathcal{C}}

\newcommand{\Mm}{\mathcal{M}}

\def\*C{{}^*\hspace*{-1pt}{\Cc}}

\def\text#1{{\rm {\rm #1}}}

\def\ul{\underline}

\begin{document}
\title[Strong group coalgebras]{Strong group coalgebras}
\author{S. Caenepeel}
\address{Faculty of Engineering,
Vrije Universiteit Brussel, B-1050 Brussels, Belgium}
\email{scaenepe@vub.ac.be}
\urladdr{http://homepages.vub.ac.be/\~{}scaenepe/}
\author{K. Janssen}
\address{Faculty of Engineering,
Vrije Universiteit Brussel, B-1050 Brussels, Belgium}
\email{krjansse@vub.ac.be}
\urladdr{http://homepages.vub.ac.be/\~{}krjansse/}

\subjclass[2000]{16W30, 16W50}

\keywords{group coalgebras, strongly graded ring, cleft extension}

\thanks{This research was supported by the research project G.0622.06 ``Deformation quantization methods
for algebras and categories with applications to quantum mechanics" from
FWO-Vlaanderen.}
\begin{abstract}
We introduce strong group coalgebras, as a generalization of strongly graded coalgebras.
We give several characterizations, and study two special types of strong group coalgebras,
namely cleft group algebras (or crossed coproduct group coalgebras) and smash
coproduct group coalgebras.
\end{abstract}
\maketitle

\section*{Introduction}\selabel{0}
Graded coalgebras were introduced by  N\u ast\u asescu and Torrecillas in
\cite{NasTor}, and further studied in \cite{DNRV}. At first glance, most results about graded
coalgebras seem to be completely similar to corresponding results about graded algebras.
However, there are some remarkable differences. For example, in \cite{NasTor},
the notion of strongly graded coalgebra is introduced. An interesting property is the
fact that strongly graded coalgebras only exist in the case where the grading group
$G$ is a finite group; we do not have such a property for strongly graded algebras.
In the coalgebra case, this property is basically a consequence of the intrinsic finiteness
that is built in the definition of a coalgebra.

Group coalgebras and Hopf group coalgebras were introduced by Turaev in \cite{Turaev}.
An algebraic study of Hopf group coalgebras was initiated in \cite{Virelizier}, and continued in
a series of papers by various authors. In \cite{CDL}, it was shown that group coalgebras
(resp. Hopf group coalgebras) are in fact coalgebras (resp. Hopf algebras) in a
well-chosen symmetric monoidal category.

Group coalgebras are a generalization of graded coalgebras, and the two notions
coincide if we work over a finite group $G$. For an infinite group $G$, group coalgebras
behave better as far as duality is concerned. The dual of a group coalgebra is a
graded algebra, and in some situations the category of comodules over the group coalgebra
is isomorphic to the category of graded modules over its dual graded algebra. These
duality results have been studied in \cite{CJW}, in the slightly more general context
of corings rather than coalgebras.

In this note, we introduce strong group coalgebras. If we work over a finite group $G$,
then strong group coalgebras correspond bijectively to strongly graded coalgebras.
However, strong group coalgebras can also exist in the case where $G$ is infinite.
The basic example is that of a cofree group coalgebra, playing the role of group algebra
in the graded algebra theory. We have several characterizations of strong group coalgebras,
see \prref{1.3} and \thref{1.7}. The most important one is perhaps the following:
for any group coalgebra, we have a pair of adjoint functors between the category of
comodules over the part of degree $e$, and the category of group comodules over the group
coalgebra; the group coalgebra is strong if and only if this adjunction is a pair of inverse
equivalences. Over a group coalgebra, one can define two types of comodules
(comparable to the situation in graded ring theory, where one has graded and
ungraded modules). In \seref{2}, we show that there is a relation between these two
types: comodules of one type are in fact comodules of the other type, but then over
the smash coproduct.

In Hopf algebra theory, cleft comodule algebras have been introduced and studied by
Doi (see \cite{Doi1,Doi2}). In \seref{3}, we study the corresponding notion of cocleft
group coalgebra. We have several equivalent characterizations of cocleft group
coalgebras, see \thref{3.4}. In the special situation where the part of degree $e$
of the group coalgebra is cocommutative, we can describe cocleft group coalgebras
using group cohomology, this is discussed in detail in \seref{4}.

\section{Strong group coalgebras}\selabel{1}
In what follows $k$ will denote a fixed field. Unadorned tensor products are meant to be taken over $k$.  By $G$ we will denote a fixed group with identity element $e$. In any category the identity morphism of an object $M$ will also be denoted by $M$. For categories $\Mm_i$, indexed by an arbitrary set $I$, we will denote the product category by $\prod_{i\in I}\Mm_i$. If all $\Mm_i$ are equal to one category $\Mm$, we will write $\Mm^I=\prod_{i\in I}\Mm$. The category of $k$-vector spaces is denoted by $\Mm_k$.

Consider a $G$-group coalgebra (or shorter, $G$-coalgebra) $\ul{C}$, i.e. a collection $\ul{C}=(C_\alpha)_{\alpha \in G}$ of $k$-vector spaces, with $k$-linear maps $\Delta_{\alpha,\beta}:C_{\alpha\beta} \to C_\alpha \ot C_\beta$ and $\varepsilon:C_e\to k$, for all $\alpha,\beta \in G$, satisfying
\begin{eqnarray*}
&&\hspace{-2cm}
(\Delta_{\alpha,\beta}\ot C_\gamma)\circ \Delta_{\alpha\beta, \gamma}= (C_\alpha\ot \Delta_{\beta,\gamma})\circ \Delta_{\alpha,\beta \gamma},\\
&&\hspace{-2cm}
(C_\alpha \ot\varepsilon)\circ \Delta_{\alpha, e}=C_\alpha= (\varepsilon \ot C_\alpha)\circ \Delta_{e,\alpha},
\end{eqnarray*}
for all $\alpha,\beta,\gamma \in G$. 
It is clear that $C_e$ is a usual $k$-coalgebra with comultiplication map $\Delta_{e,e}$ and counit map $\varepsilon$.

Given two $G$-coalgebras $\ul{C}$ and $\ul{D}$, a morphism $\ul{\varphi}:\ul{C}\to \ul{D}$ of $G$-coalgebras from $\ul{C}$ to $\ul{D}$ is a collection $\ul{\varphi}=(\varphi_\alpha:C_\alpha \to D_\alpha)_{\alpha\in G}$ of $k$-linear maps such that, $\varepsilon_{\ul{D}}\circ \varphi_e=\varepsilon_{\ul{C}}$ and
$(\varphi_\alpha\ot \varphi_\beta)\circ \Delta^{\ul{C}}_{\alpha,\beta}=\Delta_{\alpha,\beta}^{\ul{D}}\circ \varphi_{\alpha\beta},$
for all $\alpha,\beta \in G$.

Given a $G$-coalgebra $\ul{C}$ one can study two different types of modules over $\ul{C}$. 
Firstly, a $k$-vector space $M$ together with a family of $k$-linear maps $\{\rho_\alpha=\rho_\alpha^M:M\to M\ot C_\alpha\}_{\alpha\in G}$ such that $(M\ot \varepsilon) \circ \rho_e =M$ and
\begin{eqnarray*}
&&\hspace{-2cm}
(M\ot \Delta_{\alpha,\beta})\circ \rho_{\alpha\beta}= (\rho_\alpha \ot C_\beta)\circ \rho_\beta,
\end{eqnarray*}
for all $\alpha,\beta \in G$, is called a right $\ul{C}$-comodule. We will use the Sweedler-type notation $\rho_\alpha(m)=m_{[0]}\ot m_{[1,\alpha]}\in M\ot C_\alpha$, for all $m\in M$ and $\alpha\in G$, where summation is implicitly understood. Given two right $\ul{C}$-comodules $M$ and $N$, a $k$-linear map $f:M\to N$ is called right $\ul{C}$-colinear if $(f\ot C_\alpha)\circ \rho^M_\alpha=\rho_\alpha^N \circ f$, for all $\alpha\in G$. The category of right $\ul{C}$-comodules and right $\ul{C}$-colinear maps is denoted by $\Mm^{\ul{C}}$.

Secondly, a right $G$-group $\ul{C}$-comodule (or shortly, right $G$-$\ul{C}$-comodule) is a family of $k$-vector spaces $\ul{M}=(M_\alpha)_{\alpha\in G}$ equipped with a family of $k$-linear maps $\{\rho_{\alpha,\beta}=\rho_{\alpha,\beta}^{\ul{M}}:M_{\alpha\beta}\to M_\alpha \ot C_\beta\}_{\alpha,\beta \in G}$ such that $(M_\alpha \ot \varepsilon)\circ \rho_{\alpha,e}=M_\alpha$ and 
\begin{eqnarray*}
&&\hspace{-2cm}
(\rho_{\alpha,\beta}\ot C_\gamma)\circ \rho_{\alpha\beta, \gamma}= (M_\alpha\ot \Delta_{\beta,\gamma})\circ \rho_{\alpha,\beta \gamma},
\end{eqnarray*}
for all $\alpha,\beta,\gamma \in G$. We will use the notation $\rho_{\alpha,\beta}(m)=m_{[0,\alpha]}\ot m_{[1,\beta]}\in M_\alpha\ot C_\beta$, for all $m\in M_{\alpha\beta}$ and $\alpha,\beta \in G$. A right $G$-$\ul{C}$-colinear map between two right $G$-$\ul{C}$-comodules $\ul{M}$ and $\ul{N}$ is a family of $k$-linear maps $\ul{\varphi}=(\varphi_\alpha:M_\alpha \to N_\alpha)_{\alpha\in G}$ such that $(f_\alpha\ot C_\beta)\circ \rho^{\ul{M}}_{\alpha,\beta}=\rho^{\ul{N}}_{\alpha,\beta}\circ f_{\alpha\beta}$, for all $\alpha,\beta \in G$. The category of right $G$-$\ul{C}$-comodules and right $G$-$\ul{C}$-colinear maps will be denoted by $\Mm^{G,\ul{C}}$.

Given a right $G$-$\ul{C}$-comodule $\ul{M}$, it is clear that $F(\ul{M})=M_e$ is a right $C_e$-comodule with coaction map $\rho_{e,e}$. Clearly we have a functor $F:\Mm^{G,\ul{C}}\to \Mm^{C_e}$.  Given a right $C_e$-comodule $N$, we can define a right $G$-$\ul{C}$-comodule $G(N)$, by $G(N)_\alpha=N\Box_{C_e}C_\alpha$ and 
$$\rho_{\alpha,\beta}=N\Box_{C_e} \Delta_{\alpha,\beta}:N\Box_{C_e}C_{\alpha\beta}\to N\Box_{C_e} (C_\alpha \ot C_\beta)\cong (N\Box_{C_e}C_\alpha)\ot C_\beta,$$ for all $\alpha,\beta\in G$.

\begin{proposition}\prlabel{1.1}
$(F,G)$ is a pair of adjoint functors between the categories $\Mm^{G,\ul{C}}$ and $\Mm^{C_e}$.
\end{proposition}

\begin{proof}
The unit of the adjunction is given by $\ul{\eta}_{\ul{M}}=(\eta_{\ul{M},\alpha})_{\alpha\in G}:\ul{M}\to G(M_e),$ where $\eta_{\ul{M},\alpha}:M_\alpha \to M_e\Box_{C_e} C_\alpha$ is the corestriction of $\rho_{e,\alpha}:M_\alpha \to M_e\ot C_\alpha$ to $M_e\Box_{C_e} C_\alpha$, for all $\ul{M}\in \Mm^{G,\ul{C}}$.
The counit is given by the natural isomorphisms $\nu_N:N\Box_{C_e} C_e\to N, \nu_N(\sum_i n_i\ot c_i)=\sum_i \varepsilon(c_i)n_i$, for all $N\in \Mm^{C_e}$.
We still need to verify, for all $\ul{M}\in \Mm^{G,\ul{C}}$, $N\in \Mm^{C_e}$ and $\alpha \in G$, that the following diagrams
$$
\xymatrix{
M_e\ar[r]^-{\eta_{\ul{M},e}} \ar[rd]_-{=}& M_e\Box_{C_e}C_e\ar[d]^-{\nu_{M_e}} && N\Box_{C_e}C_\alpha \ar[r]^-{\eta_{GN,\alpha}}\ar[rd]_-{=} & (N\Box_{C_e}C_e)\Box_{C_e} C_\alpha\ar[d]^-{\nu_N\Box_{C_e}C_\alpha}\\
&M_e &&& N\Box_{C_e} C_\alpha\\
}$$
commute. 
Indeed, for all $m\in M_e$ we have that
\begin{eqnarray*}
&&\hspace{-2cm}
(\nu_{M_e}\circ \eta_{\ul{M},e})(m)= \nu_{M_e}(m_{[0,e]}\ot m_{[1,e]})= \varepsilon(m_{[1,e]})m_{[0,e]}=m;
\end{eqnarray*}
for $\sum_i n_i\ot c_i\in N\Box_{C_e}C_\alpha$ we have that
\begin{eqnarray*}
&&\hspace{-2cm}
\big((\nu_N\Box_{C_e}C_\alpha)\circ \eta_{GN,\alpha}\big)\Big(\sum_i n_i\ot c_i\Big)
= \sum_i \nu_N(n_i\ot c_{i(1,e)})\ot c_{i(2,\alpha)}\\
&=& \sum_i n_i\ot \varepsilon(c_{i(1,e)})c_{i(2,\alpha)}= \sum_i n_i\ot c_i.
\end{eqnarray*}
\end{proof}

\begin{definition}\delabel{1.2}
A $G$-coalgebra $\ul{C}$ is called \emph{strong} if $\Delta_{\alpha,\beta}:C_{\alpha\beta}\to C_\alpha\ot C_\beta$ is a monomorphism, for all $\alpha,\beta \in G$.
\end{definition}

\begin{proposition}\prlabel{1.3}
Let $\ul{C}$ be a $G$-coalgebra. The following assertions are equivalent:
\begin{enumerate}
\item $\ul{C}$ is a strong $G$-coalgebra;
\item for all $\alpha\in G$, $\Delta_{\alpha,\alpha^{-1}}:C_e\to C_\alpha\ot C_{\alpha^{-1}}$ is a monomorphism;
\item for all $\ul{M}\in \Mm^{G,\ul{C}}$ and $\alpha,\beta \in G$, $\rho^{\ul{M}}_{\alpha,\beta}$ is a monomorphism.
\end{enumerate}
\end{proposition}

\begin{proof}
$\ul{(1) \Rightarrow (2)}$ and $\ul{(3)\Rightarrow (1)}$ are trivial. We prove $\ul{(2)\Rightarrow (3)}$. For $\ul{M}\in  \Mm^{G,\ul{C}}$ and $\alpha,\beta \in G$, consider the diagram
$$
\xymatrix{
M_{\alpha\beta}\ar[rr]^-{\rho_{\alpha\beta,e}}\ar[d]_-{\rho_{\alpha,\beta}} && M_{\alpha\beta}\ot C_e\ar[d]^-{M_{\alpha\beta}\ot \Delta_{\beta^{-1},\beta}}\\
M_\alpha\ot C_\beta\ar[rr]_-{\rho_{\alpha\beta,\beta^{-1}}\ot C_\beta}&& M_{\alpha\beta}\ot C_{\beta^{-1}}\ot C_\beta.
}$$
By assumption $ \Delta_{\beta^{-1},\beta}$ is monic, hence also $M_{\alpha\beta}\ot  \Delta_{\beta^{-1},\beta}$ (since we work over a field $k$, $M_{\alpha\beta}$ is $k$-flat). By the counit property $\rho_{\alpha\beta,e}$ is monic. Indeed, if $\rho_{\alpha\beta,e}(m)=m_{[0,\alpha\beta]}\ot m_{[1,e]}=0$, then $m=m_{[0,\alpha\beta]}\varepsilon(m_{[1,e]})=0$. Thus $(M_{\alpha\beta}\ot  \Delta_{\beta^{-1},\beta})\circ \rho_{\alpha\beta,e}= (\rho_{\alpha\beta,\beta^{-1}}\ot C_\beta)\circ \rho_{\alpha,\beta}$ is monic. It follows that $\rho_{\alpha,\beta}$ is monic, as wanted.
\end{proof}

\begin{proposition}\prlabel{1.4} 
Let $\ul{C}$ be a strong $G$-coalgebra, and $\ul{M}\in \Mm^{G,\ul{C}}$. The following assertions are equivalent:
\begin{enumerate}
\item $\ul{M}=0$, i.e., $M_\alpha=0$ for all $\alpha \in G$;
\item $M_\alpha=0$ for some $\alpha \in G$.
\end{enumerate}
\end{proposition}
\begin{proof}
\ul{$(1)\Rightarrow (2)$} is trivial, we prove
\ul{$(2)\Rightarrow (1)$}. Take $\beta \in G$ arbitrary. We have that $\rho_{\alpha,\alpha^{-1}\beta}:M_\beta \to M_\alpha \ot C_{\alpha^{-1}\beta}$ is monic. Now $M_\alpha=0$ implies that $M_\beta=0$, and we find our result. 
\end{proof}

\begin{corollary}\colabel{1.5}
Let $\ul{C}$ be a strong $G$-coalgebra, and $\ul{\varphi}:\ul{M}\to \ul{N}$ a morphism in $\Mm^{G,\ul{C}}$. 
We have that
\begin{enumerate}
\item $\varphi_\alpha$ is injective for all $\alpha \in G$ if and only if $\varphi_\alpha$ is injective for some $ \alpha\in G$;
\item $\varphi_\alpha$ is surjective for all $\alpha \in G$ if and only if $\varphi_\alpha$ is surjective for some $ \alpha\in G$;
\item $\varphi_\alpha$ is bijective for all $\alpha \in G$ if and only if $\varphi_\alpha$ is bijective for some $ \alpha\in G$;
\end{enumerate}
\end{corollary}

\begin{proof}
This follows from \prref{1.4}.
\end{proof}

\begin{lemma}\lelabel{1.6}
Let $\ul{C}$ be a group coalgebra.
$\Mm^{G,\ul{C}}$ is an abelian category.
\end{lemma}

\begin{proof}
Let $\ul{\varphi}:\ul{M}\to \ul{N}$ be a morphism in $\Mm^{G,\ul{C}}$. Write $K_\alpha=\Ker(\varphi_\alpha)$ and $I_\alpha=\im(\varphi_\alpha)$. For $m\in M_{\alpha\beta}$, we have that
\begin{eqnarray*}
&&\hspace{-2cm}
(\varphi_\alpha\ot C_\beta)(\rho^{\ul{M}}_{\alpha,\beta}(m))=\varphi_\alpha(m_{[0,\alpha]})\ot m_{[1,\beta]}=\rho^{\ul{N}}_{\alpha,\beta}(\varphi_{\alpha\beta}(m)).
\end{eqnarray*} 
This implies that $\rho^{\ul{N}}_{\alpha,\beta}(I_{\alpha\beta})\subset I_\alpha\ot C_\beta$. Also if $m\in K_{\alpha\beta}$, then $\rho^{\ul{M}}_{\alpha,\beta}(m)\in \Ker(\varphi_\alpha\ot C_\beta)=K_\alpha\ot C_\beta$ (the last equality holds since $C_\beta$ is $k$-flat). So $\rho^{\ul{M}}_{\alpha,\beta}(K_{\alpha\beta})\subset K_\alpha \ot C_\beta$. This shows that $\ul{K}=(K_\alpha)_{\alpha\in G},\ul{I}=(I_\alpha)_{\alpha\in G}\in \Mm^{G,\ul{C}}$.
\end{proof}

\begin{theorem}\thlabel{1.7}
Let $\ul{C}$ be a $G$-coalgebra. The following assertions are equivalent:
\begin{enumerate}
\item $\ul{C}$ is a strong $G$-coalgebra;
\item $(F,G)$ is a pair of inverse equivalences between $\Mm^{G,\ul{C}}$ and $\Mm^{C_e}$;
\item $\Delta_{\alpha,\beta}$ corestricts to an isomorphism $C_{\alpha\beta}\to C_\alpha\Box_{C_e}C_\beta$, for all $\alpha,\beta\in G$.
\end{enumerate}
\end{theorem}

\begin{proof}
\ul{$(1)\Rightarrow (2)$}. 
It suffices to show that, for all $\alpha \in G$ and $\ul{M}\in \Mm^{G,\ul{C}}$, $\eta_{\ul{M},\alpha}:M_\alpha \to M_e\Box_{C_e}C_\alpha$ is bijective. Since $\eta_{\ul{M},e}:M_e\to M_e\Box_{C_e} C_e$ is bijective, the result follows from \coref{1.5}.\\
\ul{$(2)\Rightarrow (3)$}. Consider the $\sigma$-suspension
$\ul{C}_\sigma=(C_{\sigma\alpha})_{\alpha\in G}$ of $\ul{C}$ (see also \cite[Section 3]{NasTor}).
The right $G\hbox{-}\ul{C}$-comodule structure on $\ul{C}_\sigma$ is given by the maps
$\Delta_{\sigma\alpha,\beta}:\ C_{\sigma\alpha\beta}\to C_{\sigma\alpha}\ot C_\beta$. Since
$(F,G)$ is a pair of inverse equivalences, 
$\eta_{\ul{C(\alpha)},\beta}=\Delta_{\alpha,\beta}:C_{\alpha\beta}\to C_\alpha\Box_{C_e}C_\beta$ is an isomorphism, for all $\alpha,\beta \in G$.\\
\ul{$(3)\Rightarrow (1)$}. 
It follows from (3) that all maps $\Delta_{\alpha,\beta}:C_{\alpha\beta}\to C_\alpha \ot C_\beta$ are monomorphic.
\end{proof}

Let $\ul{C}$ be a $G$-coalgebra. We can construct a graded $k$-algebra $R$ as follows
(see \cite[Sec. 5]{CJW}). For $\alpha\in G$, let $R_\alpha=C^*_{\alpha^{-1}}$. $R=
\oplus_{\alpha\in G} R_\alpha$ is a $G$-graded $k$-algebra, with multiplication
$$(f\#g)(c)=f(c_{(2,\alpha^{-1})})g(c_{(1,\beta^{-1})}),$$
for $f\in R_\alpha$, $g\in R_\beta$ and $c\in R_{\beta^{-1}\alpha^{-1}}$. The unit element is
$\varepsilon\in R_e$. The multiplication $m_{\alpha,\beta}:\ R_\alpha\ot R_\beta\to R_{\alpha\beta}$
is the composition of the dual of the comultiplication map $\Delta_{\beta^{-1},\alpha^{-1}}$ and
the canonical inclusion $R_\alpha\ot R_\beta=C^*_{\alpha^{-1}}\ot C^*_{\beta^{-1}}\to
(C_{\beta^{-1}}\ot C_{\alpha^{-1}})^*$. If $\ul{C}$ is homogeneously finite, this means that
every $C_\alpha$ is finite dimensional, then this canonical inclusion is an isomorphism,
and then $m_{\alpha,\beta}$ is surjective if and only if $\Delta_{\beta^{-1},\alpha^{-1}}$
is injective. By definition $R$ is strongly graded if and only if the maps $m_{\alpha,\beta}$
are surjective (see for example \cite{NVO}). This proves the following result.

\begin{proposition}\prlabel{1.8}
Let $\ul{C}$ be a homogeneously finite $G$-coalgebra. Then $\ul{C}$ is strong if and only if
$R=\oplus_{\alpha\in G} C^*_\alpha$ is a strongly graded $k$-algebra.
\end{proposition}

\section{Group comodules and the smash coproduct}\selabel{2}
Let $\ul{C}$ be a $G$-coalgebra. We introduce the so-called \emph{smash coproduct $G$-coalgebra} $\ul{C}\rtimes kG$ of $\ul{C}$ and $kG$ as follows. For $\alpha \in G$ we put $(\ul{C}\rtimes kG)_\alpha=C_\alpha \rtimes kG=C_\alpha \ot kG$. For $\alpha,\beta\in G$ we define $\Delta_{\alpha,\beta}:C_{\alpha\beta}\rtimes kG \to (C_{\alpha}\rtimes kG)\ot (C_{\beta}\rtimes kG)$ by
$$\Delta_{\alpha,\beta}(c\rtimes \sigma)=(c_{(1,\alpha)}\rtimes \beta \sigma) \ot (c_{(2,\beta)}\rtimes \sigma),$$
for $c\in C_{\alpha\beta}$ and $\sigma \in G$. We define $\varepsilon:C_e\rtimes kG \to k$ by
$\varepsilon(c\rtimes \sigma)=\varepsilon(c)$,
for $c\in C_e$ and $\sigma \in G$.
\begin{lemma}
$\ul{C}\rtimes kG$ is a $G$-coalgebra.
\end{lemma}
\begin{proof}
Let $\alpha,\beta,\gamma,\sigma\in G$ and take $c\in C_{\alpha \beta \gamma}$. We then have that
\begin{eqnarray*}
&&\hspace{-1cm}
\big((\Delta_{\alpha,\beta}\ot (C_\gamma\rtimes kG))\circ \Delta_{\alpha\beta,\gamma}\big)(c\rtimes \sigma)
= \Delta_{\alpha,\beta}(c_{(1,\alpha\beta)}\rtimes \gamma\sigma) \ot (c_{(2,\gamma)}\rtimes \sigma)\\
&=& (c_{(1,\alpha)}\rtimes \beta\gamma\sigma) \ot (c_{(2,\beta)}\rtimes \gamma\sigma) \ot (c_{(3,\gamma)}\rtimes \sigma)\\
&=& (c_{(1,\alpha)}\rtimes \beta\gamma\sigma) \ot \Delta_{\beta,\gamma}(c_{(2,\beta\gamma)}\rtimes\sigma) \\
&=& \big(((C_\alpha \rtimes kG)\ot \Delta_{\beta,\gamma})\circ \Delta_{\alpha,\beta\gamma}\big)(c\rtimes \sigma).
\end{eqnarray*}
For $\alpha,\sigma \in G$ and $c\in C_\alpha$ we have
\begin{eqnarray*}
&&\hspace{-2cm}
\big( ((C_\alpha \rtimes kG)\ot \varepsilon)\circ \Delta_{\alpha,e}\big)(c\rtimes \sigma)
= \varepsilon(c_{(2,e)} \rtimes \sigma) c_{(1,\alpha)} \rtimes \sigma \\
&=&  \varepsilon(c_{(2,e)}) c_{(1,\alpha)} \rtimes \sigma = c\rtimes \sigma;\\
&&\hspace{-2cm}
\big( (\varepsilon \ot (C_\alpha \rtimes kG))\circ \Delta_{e,\alpha}\big)(c\rtimes \sigma)
= \varepsilon(c_{(1,e)} \rtimes \alpha\sigma) c_{(2,\alpha)} \rtimes \sigma \\
&=& \varepsilon(c_{(1,e)} ) c_{(2,\alpha)} \rtimes \sigma = c\rtimes \sigma.
\end{eqnarray*}
\end{proof}

\begin{proposition}\prlabel{2.2}
Let $\ul{C}$ be a $G$-coalgebra. The categories $\Mm^{\ul{C}\rtimes kG}$ and $\Mm^{G,\ul{C}}$ are isomorphic.
\end{proposition}

\begin{proof}
Take $\ul{M}\in \Mm^{G,\ul{C}}$. Let $M=\bigoplus_{\alpha\in G}M_\alpha$, and define
$\rho_\beta:M \to M\ot (C_\beta \rtimes kG)$ as follows: for $m\in M_\sigma$, let
$$\rho_\beta(m)=m_{[0,\sigma \beta^{-1}]}\ot (m_{[1,\beta]}\rtimes \sigma^{-1}).$$
Then $M\in \Mm^{\ul{C}\rtimes kG}$. Indeed, for $m\in M_\sigma$ we have that
$$
((M\ot \varepsilon)\circ \rho_e)(m)
= (M\ot \varepsilon)(m_{[0,\sigma]}\ot (m_{[1,e]} \rtimes \sigma^{-1}))
= \varepsilon(m_{[1,e]})m_{[0,\sigma]}=m.
$$
We still need to verify the commutativity of the diagram
\begin{eqnarray}\eqlabel{diagram}
\xymatrix{
M\ar[rr]^-{\rho_{\alpha\beta}}\ar[d]_-{\rho_\beta} && M\ot(C_{\alpha\beta} \rtimes kG)\ar[d]^-{M\ot \Delta_{\alpha,\beta}}\\
M\ot(C_\beta \rtimes kG)\ar[rr]_-{\rho_\alpha \ot (C_\beta\rtimes kG)} && M\ot(C_\alpha \rtimes kG) \ot (C_\beta \rtimes kG).
}
\end{eqnarray}
Take $m\in M_\sigma$. We then have that
\begin{eqnarray*}
&&\hspace{-1cm}
\big( (M\ot \Delta_{\alpha,\beta}) \circ \rho_{\alpha\beta}\big)(m)
= (M\ot \Delta_{\alpha,\beta})(m_{[0,\sigma \beta^{-1}\alpha^{-1}]}\ot (m_{[1,\alpha\beta]}\rtimes \sigma^{-1}))\\
&=& m_{[0,\sigma \beta^{-1}\alpha^{-1}]}\ot (m_{[1,\alpha]}\rtimes \beta \sigma^{-1}) \ot (m_{[2,\beta]}\rtimes  \sigma^{-1}) \\
&=& \rho_\alpha(m_{[0,\sigma \beta^{-1}]}) \ot (m_{[1,\beta]}\rtimes  \sigma^{-1}) 
= \big((\rho_\alpha \ot (C_\beta\rtimes kG))\circ \rho_\beta \big)(m),
\end{eqnarray*}
as needed. Conversely, assume that $M\in \Mm^{\ul{C}\rtimes kG}$. For each $\alpha\in G$ we have a map $\rho_\alpha:M\to M\ot (C_\alpha \rtimes kG)$. $\rho_e$ turns $M$ into a right $C_e\rtimes kG$-comodule. Since $p_e:C_e\rtimes kG\to kG,p_e(c\rtimes \sigma)=\varepsilon(c)\sigma^{-1}$ is a coalgebra map, we have a right $kG$-coaction 
$(M\ot p_e)\circ \rho_e:M\to M\ot kG$
on $M$. This makes $M$ into a $G$-graded $k$-vector space
$M=\bigoplus_{\sigma\in G}M_\sigma$,
where $m\in M_\sigma$ if and only if $\rho_e(m)\in M\ot (C_e\rtimes \sigma^{-1})$. 
For $m\in M$, we introduce the notation
$$\rho_\alpha(m)= \sum_{\sigma\in G}m^\sigma_{[0]}\ot m^\sigma_{[1,\alpha]}\rtimes \sigma^{-1}\in M\ot(C_\alpha\rtimes kG).$$
With this notation we have that $m\in M_\tau$ if and only if $m^\sigma_{[0]}\ot m^\sigma_{[1,e]}=0$, for all $\sigma\neq \tau$.
We have that the diagram \equref{diagram} commutes. Let us write this down explicitely.
For all $m\in M$ and $\alpha,\beta \in G$, we have that
\begin{eqnarray*}
&&\hspace{-1cm}
\big( (M\ot \Delta_{\alpha,\beta}) \circ \rho_{\alpha\beta}\big)(m)
= (M\ot \Delta_{\alpha,\beta})\Big( \sum_\sigma m^\sigma_{[0]}\ot (m^\sigma_{[1,\alpha\beta]} \rtimes \sigma^{-1})\Big)\\
&=& \sum_\sigma m^\sigma_{[0]}\ot (m^\sigma_{[1,\alpha\beta](1,\alpha)} \rtimes \beta\sigma^{-1}) \ot
 (m^\sigma_{[1,\alpha\beta](2,\beta)} \rtimes \sigma^{-1})
\end{eqnarray*}
equals
\begin{eqnarray*}
&&\hspace{-1cm}
\big((\rho_\alpha \ot (C_\beta\rtimes kG))\circ \rho_\beta \big)(m)
= (\rho_\alpha \ot (C_\beta\rtimes kG))\Big(\sum_{\sigma}m^\sigma_{[0]}\ot (m^\sigma_{[1,\beta]}\rtimes \sigma^{-1}) \Big)\\
&=&\sum_{\sigma,\tau}(m^\sigma_{[0]})^\tau_{[0]}\ot ((m^\sigma_{[0]})^\tau_{[1,\alpha]}\rtimes \tau^{-1}) \ot (m^\sigma_{[1,\beta]}\rtimes \sigma^{-1}).
\end{eqnarray*}
We will refer to this equation as $(\star)$. Before we end the proof of \prref{2.2}, we state and prove two Lemmas.

\begin{lemma}\lelabel{2.3}
Let $\omega\in G$. The following assertions are equivalent:
\begin{enumerate}
\item[\emph{(i)}]$m\in M_\omega$;
\item[\emph{(ii)}]$\rho_\alpha(m)\in M\ot (C_\alpha\rtimes \omega^{-1})$, for all $\alpha\in G$;
\item[\emph{(iii)}]$m^\sigma_{[0]}\ot m^\sigma_{[1,\alpha]}=0$, for all $\omega\neq \sigma,\alpha\in G$.
\end{enumerate}
\end{lemma}
\begin{proof}
\ul{$(ii)\Leftrightarrow (iii)$} and \ul{$(iii)\Rightarrow (i)$} are obvious. We prove \ul{$(i)\Rightarrow (iii)$}. Take $\beta=e$ in $(\star)$. We view the equality $(\star)$ as an equality in $\bigoplus_{\sigma,\tau\in G}(M\ot C_\alpha \ot C_\beta)e_{\sigma,\tau}$. 
For $\sigma \neq \omega$, the $(\sigma^{-1},\sigma^{-1})$-component of $(\star)$ comes out as
\begin{eqnarray*}
&&\hspace{-1cm}
0= (m^\sigma_{[0]})^\sigma_{[0]}\ot (m^\sigma_{[0]})^\sigma_{[1,\alpha]} \ot m^\sigma_{[1,e]}
=m^\sigma_{[0]}\ot m^\sigma_{[1,\alpha](1,\alpha)} \ot m^\sigma_{[1,\alpha](2,e)}.
\end{eqnarray*}
If we apply $\varepsilon$ to the third tensor factor, then we find
$0=m^\sigma_{[0]}\ot m^\sigma_{[1,\alpha]},$
as needed.
\end{proof}

\begin{lemma}\lelabel{2.4}
Let $\omega\in G$ and $m\in M_\omega$. Then, for all $\beta\in G$,
$$\rho_\beta(m)\in M_{\omega\beta^{-1}}\ot C_\beta \rtimes \omega^{-1}.$$
\end{lemma}

\begin{proof}
By  \leref{2.3} $(\star)$ takes the form
\begin{eqnarray*}
&&\hspace{-2cm}
m^\omega_{[0]}\ot (m^\omega_{[1,\alpha\beta](1,\alpha)} \rtimes \beta\omega^{-1}) \ot
 (m^\omega_{[1,\alpha\beta](2,\beta)} \rtimes \omega^{-1})\\
 &=& \sum_{\tau}(m^\omega_{[0]})^\tau_{[0]}\ot ((m^\omega_{[0]})^\tau_{[1,\alpha]}\rtimes \tau^{-1}) \ot (m^\omega_{[1,\beta]}\rtimes \omega^{-1}).
\end{eqnarray*}
Fix $\tau\neq \omega\beta^{-1}$, and take the $(\tau^{-1},\omega^{-1})$-component of both sides. This gives
$$0=(m^\omega_{[0]})^\tau_{[0]}\ot (m^\omega_{[0]})^\tau_{[1,\alpha]}\ot m^\omega_{[1,\beta]},$$
and 
$$(\rho_\alpha \ot (C_\beta\rtimes kG))(\rho_\beta (m))\in M\ot(C_\alpha \rtimes \beta \omega^{-1})\ot (C_\beta \rtimes \omega^{-1}).$$
Hence $\rho_\beta(m)\in M_{\omega\beta^{-1}}\ot C_\beta \rtimes \omega^{-1}$.
\end{proof}

Now we complete the proof of \prref{2.2}. For $\alpha,\beta\in G$, and $m\in M_{\alpha\beta}$, we have
$$\rho_\beta(m)=m_{[0]}^{\alpha\beta}\ot m_{[1,\beta]}^{\alpha\beta} \rtimes \beta^{-1}\alpha^{-1}\in M_\alpha\ot C_\beta \rtimes \beta^{-1}\alpha^{-1}.$$
We define $\rho_{\alpha,\beta}:M_{\alpha\beta}\to M_\alpha\ot C_\beta$ by $\rho_{\alpha,\beta}(m)=m_{[0]}^{\alpha\beta}\ot m_{[1,\beta]}^{\alpha\beta}$.
Let us show that $\ul{M}=(M_\alpha)_{\alpha\in G}$ satisfies the coassociativity condition. If $m\in M_\omega$, then $\rho_\alpha(m)=m^\omega_{[0]}\ot m^\omega_{[1,\alpha]} \rtimes \omega^{-1}$.
Now take $m\in M_{\alpha\beta\gamma}$. Then
\begin{eqnarray*}
&&\hspace{-1cm}
((M\ot \Delta_{\beta,\gamma})\circ \rho_{\beta\gamma})(m)\\
&=& m_{[0]}^{\alpha\beta\gamma}\ot (m^{\alpha\beta\gamma}_{[1,\beta\gamma](1,\beta)} \rtimes \beta^{-1}\alpha^{-1}) \ot (m^{\alpha\beta\gamma}_{[1,\beta\gamma](2,\gamma)} \rtimes \gamma^{-1}\beta^{-1}\alpha^{-1})
\end{eqnarray*}
is equal to
\begin{eqnarray*}
&&\hspace{-1cm}
((\rho_\beta \ot C_\gamma)\circ \rho_{\gamma})(m)\\
&=& (m_{[0]}^{\alpha\beta\gamma})_{[0]}^{\alpha\beta} \ot ((m_{[0]}^{\alpha\beta\gamma})_{[1,\beta]}^{\alpha\beta} \rtimes \beta^{-1}\alpha^{-1}) \ot (m^{\alpha\beta\gamma}_{[1,\gamma]} \rtimes \gamma^{-1}\beta^{-1}\alpha^{-1}).
\end{eqnarray*}
Hence 
\begin{eqnarray*}
&&\hspace{-1cm}
((M_\alpha\ot \Delta_{\beta,\gamma})\circ \rho_{\alpha,\beta\gamma})(m)
=(M_\alpha\ot \Delta_{\beta,\gamma})(m_{[0]}^{\alpha\beta\gamma} \ot m^{\alpha\beta\gamma}_{[1,\beta\gamma]})\\
&=& (m_{[0]}^{\alpha\beta\gamma})_{[0]}^{\alpha\beta} \ot (m_{[0]}^{\alpha\beta\gamma})_{[1,\beta]}^{\alpha\beta} \ot m^{\alpha\beta\gamma}_{[1,\gamma]}
= ((\rho_{\alpha,\beta}\ot C_\gamma)\circ \rho_{\alpha\beta,\gamma})(m).
\end{eqnarray*}
Take $m\in M_\beta$. Then $\rho_e(m)=m^\beta_{[0]}\ot (m^\beta_{[1,e]}\rtimes \beta^{-1})$, and
$\rho_{\beta,e}(m)=m^\beta_{[0]}\ot m^\beta_{[1,e]}$. It follows that
$$((M\ot \varepsilon)\circ \rho_{\beta,e})(m)=((M\ot\varepsilon)\circ \rho_e)(m)=m,$$
so the counit property is also satisfied.
\end{proof}

Combining Proposition \ref{pr:1.1} and \ref{pr:2.2}, we obtain a pair of adjoint functors $(F',G')$ between
the categories $\Mm^{\ul{C}\rtimes kG}$ and $\Mm^{C_e}$. For $(M,(\rho_\alpha)_{\alpha\in G)
\in \Mm^{\ul{C}\rtimes kG}}$,
$$F'(M)=\{m\in M~|~\rho_e(m)\in M\ot (C_e\rtimes e)\}.$$
For $N\in \Mm^{C_e}$, $G'(N)=(\oplus_{\alpha\in G} N\sq_{C_e} C_\alpha, (\rho_\alpha)_{\alpha
\in G})$, with
$$\rho_\beta:\ \oplus_{\alpha\in G} N\sq_{C_e} C_\alpha\to \oplus_{\alpha\in G} (N\sq_{C_e} C_\alpha)
\otimes (C_\beta\rtimes kG)$$
defined as follows: for $\sum_i n_i\ot c_i\in N\sq_{C_e} C_\alpha$,
$$\rho_\beta(\sum_i n_i\ot c_i)=
\sum_i (n_i\ot c_{i(1,\alpha\beta^{-1})}\ot (c_{i(2,\beta)}\rtimes \sigma^{-1}.$$
It follows from \thref{1.7} and \prref{2.2} that $(F',G')$ is a pair of inverse equivalences if
and only if $\ul{C}$ is a strong $G$-coalgebra.

\section{Crossed coproducts and cocleft $G$-coalgebras}\selabel{3}
Let $k\langle G\rangle$ be the $G$-coalgebra defined by
$k\langle G\rangle_\sigma=kp_\sigma$, 
$\Delta_{\sigma,\tau}(p_{\sigma \tau})=p_\sigma \ot p_\tau$ and
$\varepsilon(p_e)=1$, for all $\sigma,\tau \in G$.
$k\langle G\rangle$ is even a Hopf $G$-coalgebra: every $kp_\sigma$ is a $k$-algebra (isomorphic to $k$), and the antipode maps $S_\sigma:\ kp_\sigma\to kp_{\sigma^{-1}}$ are 
given by $S_\sigma(p_\sigma)=p_{\sigma^{-1}}$. \\

Let $C$ be a coalgebra. Suppose that we have a weak $G$-action on $C$,
this is a collection of $k$-coalgebra maps $\lambda=\{\lambda_\alpha~|~\alpha\in G\}$.
Assume moreover that we have a collection of $k$-linear maps
$f=\{f_{\alpha,\beta}~|~\alpha,\beta\in G\}\subset C^*$. We assume that $f_{e,e}$ has a convolution
inverse $g_{e,e}$ and that
\begin{equation}\eqlabel{3.1.1}
\lambda_e(c)=f_{e,e}(c_{(1)})c_{(2)}g_{e,e}(c_{(3)}).
\end{equation}
for all $c\in C$. For $\alpha,\beta\in G$, consider the maps
$$\delta_{\alpha,\beta}:\ C\to C\ot C,~~\delta_{\alpha,\beta}(c)
=c_{(1)}\ot\lambda_\alpha(c_{(2)})f_{\alpha,\beta}(c_{(3)}).$$
Now let $C\rtimes k\lan G\ran=(C\rtimes p_\alpha)_{\alpha\in G}$, with
comultiplication and counit maps given by
$$\Delta_{\alpha,\beta}(c\rtimes p_{\alpha\beta})=(c_{(1)}\rtimes p_\alpha)\ot (\lambda_\alpha(c_{(2)})f_{\alpha,\beta}(c_{(3)})\rtimes p_\beta)$$
and
$$\varepsilon(c\rtimes p_e)=g_{e,e}(c),$$
for all $\alpha,\beta\in G$ and $c\in C$. Straightforward computations now show the following
result.

\begin{proposition}\prlabel{3.1}
Let $\lambda$ be a weak $G$-action on a coalgebra $C$, and $f$ a collection of maps
satisfying \equref{3.1.1}. Then $C\rtimes k\lan G\ran$ is a $G$-coalgebra if and only
if the following conditions are satisfied, for all $c\in C$ and $\alpha,\beta,\gamma\in G$:
\begin{eqnarray}
f_{e,\alpha}= f_{e,e}&;&f_{\alpha,e}=f_{e,e}\circ\lambda_{\alpha};\eqlabel{CU}\\
f_{\beta,\gamma}(\lambda_\alpha(c_{(1)}))f_{\alpha,\beta\gamma}(c_{(2)})
&=& f_{\alpha,\beta}(c_{(1)})f_{\alpha\beta,\gamma}(c_{(2)});\eqlabel{C}\\
(\lambda_\beta \circ \lambda_\alpha)(c_{(1)})f_{\alpha,\beta}(c_{(2)})
&=& f_{\alpha,\beta}(c_{(1)})\lambda_{\alpha\beta}(c_{(2)})\eqlabel{TC}.
\end{eqnarray}
\end{proposition}

If $f$ satisfies the conditions of \prref{3.1}, and every $f_{\alpha,\beta}$ has a convolution inverse $g_{\alpha,\beta}$, then $f$ is called a factor set, and
$C\rtimes k\lan G\ran$ is called a crossed coproduct $G$-coalgebra.
$f$ is called
normalized if $f_{\alpha,e}=f_{e,\alpha}=\varepsilon_C$, for all $\alpha\in G$. Then
$\lambda_e=C$, and the counit of $C\rtimes k\lan G\ran$ is given by the formula
$\varepsilon(c\rtimes p_e)=\varepsilon_C(c)$.\\
If $f_{\alpha,\beta}=\varepsilon_C$, for all $\alpha,\beta\in G$, then we call
$C\rtimes k\lan G\ran$ a smash coproduct $G$-coalgebra. If, in addition,
$\lambda_\alpha=C$,
then $C\rtimes k\lan G\ran$ is the cofree $G$-coalgebra $C\lan G\ran$ introduced in \cite{CJW}.
In particular, if $C=k$, then we recover the $G$-coalgebra $k\lan G\ran$ introduced
at the beginning of this Section.\\

We will now show that the factor set $f$ can be chosen in such a way that it is normalized
To this end, we will apply the following construction. Let $\ul{C}$ be a $G$-coalgebra,
and $\ul{\varphi}:\ \ul{C}\to \ul{D}$ an isomorphism in $(\Mm_k)^G$. Then we can define
a $G$-coalgebra structure on $\ul{D}$ such that $\ul{\varphi}$ becomes an isomorphism
of $G$-coalgebras: the structure maps on $\ul{D}$ are $\varepsilon'=
\varepsilon\circ\varphi_e^{-1}:\ D_e\to k$ and
\begin{equation}\eqlabel{3.2.1}
\Delta'_{\alpha,\beta}=(\varphi_\alpha\ot\varphi_\beta)\circ \Delta_{\alpha,\beta}
\circ \varphi_{\alpha\beta}^{-1}:\ D_{\alpha\beta}\to D_\alpha\ot D_\beta.
\end{equation}

\begin{proposition}\prlabel{3.2}
Let $G$ be a group acting weakly on a coalgebra $C$, and $f$ a factor set.
Then there exists a set of $k$-coalgebra maps $\{\lambda'_{\alpha}~|~
\alpha\in G\}\subset \End(C)$ and a normalized factor set $f'$ such that
$C\rtimes_f k\lan G\ran \cong C\rtimes_{f'} k\lan G\ran $
as $G$-coalgebras.
\end{proposition}

\begin{proof}
For every $\alpha\in G$, consider the isomorphism $\varphi_\alpha:\
C\rtimes p_\alpha\to C\rtimes p_\alpha$, defined as follows: $\varphi_\alpha=C\rtimes p_\alpha$
if $\alpha\neq e$ and
$\varphi_e(c\rtimes p_e)=c_{(1)}g_{e,e}(c_{2)})\rtimes p_e$.
The inverse of $\varphi_e$ is defined by the formula
$\varphi_e^{-1}(c\rtimes p_e)=c_{(1)}f_{e,e}(c_{2)})\rtimes p_e$.
Applying the above construction, we find a new $G$-coalgebra structure on $C\rtimes k\lan G\ran $.
The new counit $\varepsilon$ is defined as follows:
$$\varepsilon'(c\rtimes p_e)=g_{e,e}(c_{(1)}f_{e,e}(c_{(2)}))=\varepsilon(c).$$
We compute the new comultiplication maps $\delta'_{\alpha,\beta}:\
C\to C\ot C$ using \equref{3.2.1}. Clearly $\delta'_{\alpha,\beta}=\delta_{\alpha,\beta}$
if $\alpha\neq e$, $\beta\neq e$ and $\alpha\beta\neq e$. For $\alpha\neq e$, we compute
\begin{eqnarray*}
\delta'_{e,\alpha}(c)
&=&(\varphi_e\ot C)(\delta_{e,\alpha}(c))
= (\varphi_e\ot C)\bigl(c_{(1)}\ot \lambda_e(c_{(2)})f_{e,\alpha}(c_{(3)})\bigr)\\
&\equal{(\ref{eq:3.1.1},\ref{eq:CU})}&
c_{(1)}g_{e,e}(c_{(2)})\ot f_{e,e}(c_{(3)})c_{(4)}=c_{(1)}\ot c_{(2)};\\
\delta'_{\alpha,e}(c)&=&(C\ot \varphi_e)(\delta_{\alpha,e}(c))
=(C\ot \varphi_e)(c_{(1)}\ot \lambda_\alpha(c_{(2)})f_{\alpha,e}(c_{(3)}))\\
&\equal{\equref{CU}}& c_{(1)}\ot \lambda_\alpha(c_{(2)}) g_{e,e}(\lambda_\alpha(c_{(3)}))
f_{e,e}(\lambda_\alpha(c_{(4)}))\\
&=& c_{(1)}\ot \lambda_\alpha(c_{(2)}) \varepsilon(\lambda_\alpha(c_{(3)}))
=c_{(1)}\ot \lambda_{\alpha}(c_{(2)});\\
\delta'_{\alpha,\alpha^{-1}}(c)&=&\delta_{\alpha,\alpha^{-1}}(\varphi_e^{-1}(c))
=\delta_{\alpha,\alpha^{-1}}(c_{(1)}f_{e,e}(c_{(2)}))\\
&=& c_{(1)}\ot \lambda_\alpha(c_{(2)})f_{\alpha,\alpha^{-1}}(c_{(3)})f_{e,e}(c_{(4)});\\
\delta'_{e,e}(c)
&=&(\varphi_e\ot \varphi_e)(\delta_{e,e}(c_{(1)})f_{e,e}(c_{(2)}))\\
&=&(\varphi_e\ot \varphi_e)(c_{(1)}\ot\lambda_e(c_{(2)})f_{e,e}(c_{(3)})f_{e,e}(c_{(4)}))\\
&=&c_{(1)}g_{e,e}(c_{(2)})\ot\lambda_e(c_{(3)})g_{e,e}(\lambda_e(c_{(4)}))
f_{e,e}(c_{(5)})f_{e,e}(c_{(6)})\\
&\equal{\equref{3.1.1}}&c_{(1)}g_{e,e}(c_{(2)})\ot f_{e,e}(c_{(3)})c_{(4)} g_{e,e}(c_{(5)})f_{e,e}(c_{(6)})
g_{e,e}(c_{(7)})\\
&&~~~~g_{e,e}(c_{(8)})f_{e,e}(c_{(9)})f_{e,e}(c_{(10)})
= c_{(1)}\ot c_{(2)}.
\end{eqnarray*}
It follows that $C\rtimes_f k\lan G\ran\cong C\rtimes_{f'} k\lan G\ran$ with
$$\lambda'_e=C;~~\lambda'_\alpha=\lambda_\alpha~~{\rm if~}\alpha\neq e;$$
$$f'_{\alpha,e}=f'_{e,\alpha}=\varepsilon_C,~~{\rm for~all~}\alpha\in G;$$
$$f'_{\alpha,\alpha^{-1}}= f_{\alpha,\alpha^{-1}}*f_{e,e}~~{\rm if~}\alpha\neq e;$$
$$f'_{\alpha,\beta}=f_{\alpha,\beta}~~{\rm if~}\alpha\neq e,~\beta\neq e~{\rm and}~
\alpha\beta\neq e.$$
\end{proof}

Let $\ul{C}$ be a $G$-coalgebra.
In the sequel, we will consider morphisms $\ul{u}:\ \ul{C}\to k\langle G\rangle$ in 
the category $(\Mm_k)^G$. Such a morphism is given by a collection of maps
$\{u_\alpha\in C_\alpha^*~|~\alpha\in G\}$. The $\alpha$-component of $\ul{u}$
then sends $c\in C_\alpha$ to $u_\alpha(c)p_\alpha$. We say $\ul{u}$ is convolution
invertible if there exists a collection of maps
$\{v_\alpha\in C_{\alpha^{-1}}^*~|~\alpha\in G\}$ such that
\begin{equation}\eqlabel{3.2.2}
u_\alpha(c_{(1,\alpha)})v_\alpha(c_{(2,\alpha^{-1})})=v_\alpha(c_{(1,\alpha^{-1})})u_\alpha(c_{(2,\alpha)})=\varepsilon(c),
\end{equation}
for all $c\in C_e$.
If $\ul{u}$ is a morphism of $G$-coalgebras, then $\ul{u}$ is convolution invertible:
it suffices to take $v_{\alpha}=u_{\alpha^{-1}}$.

\begin{definition}\delabel{3.3}
A $G$-coalgebra $\ul{C}$ is called \emph{cocleft} over $C_e$ if there exists a \emph{convolution invertible} morphism $\ul{u}:\ul{C}\to k\langle G\rangle$ in $(\Mm_k)^G$.
\end{definition}

\begin{theorem}\thlabel{3.4}
For a $G$-coalgebra $\ul{C}$, the following conditions are equivalent.
\begin{enumerate}
\item $\ul{C}$ is cocleft;
\item $\ul{C}$ is isomorphic to a crossed coproduct $G$-coalgebra $C\rtimes_f k\lan G\ran$;
\item $\ul{C}$ is isomorphic to a crossed coproduct $G$-coalgebra $C\rtimes_f k\lan G\ran$,
with $f$ normalized;
\item $\ul{C}$ is a strong $G$-coalgebra, and every $C_\alpha$ is isomorphic to $C_e$
as a left $C_e$-comodule.
\end{enumerate}
\end{theorem}

\begin{proof}
$\ul {(1)\Rightarrow (2)}$. For $\alpha,\beta\in G$, we define $\lambda_\alpha:\
C_e\to C_e$ and $f_{\alpha,\beta}:\ C_e\to k$ by the formulas
\begin{eqnarray}\eqlabel{3.4.1}
\lambda_\alpha(c)&=&u_\alpha(c_{(1,\alpha)})c_{(2,e)}v_\alpha(c_{(3,\alpha^{-1})}),\\
\eqlabel{3.4.2}
f_{\alpha,\beta}(c)&=&u_\alpha(c_{(1,\alpha)})u_\beta(c_{(2,\beta)})v_{\alpha\beta}(c_{(3,\beta^{-1}\alpha^{-1})}).
\end{eqnarray}
It is easy to verify that the $\lambda_\alpha$ are coalgebra maps, and that $f$ is a factor set.
Clearly the maps $f_{\alpha,\beta}$ are convolution invertible. Now define
$\ul{\varphi}: \ul{C}\to C_e\rtimes k\lan G\ran$ by 
$\varphi_\alpha(c)=c_{(1,e)}u_{\alpha}(c_{(2,\alpha)})\rtimes p_\alpha$.
$\ul{\varphi}$ is an isomorphism of $G$-coalgebras, with inverse given by 
$\varphi_\alpha^{-1}(c\rtimes p_\alpha)= c_{(1,\alpha)}v_\alpha(c_{(2,\alpha^{-1})})$.\\
$\ul {(2)\Rightarrow (3)}$ follows immediately from \prref{3.2}.\\
$\ul {(3)\Rightarrow (4)}$. Let $\ul{C}=C\rtimes k\lan G\ran$ be a crossed coproduct
$G$-coalgebra. It is easy to  see that $\delta_{e,e}=\delta_{e,\alpha}$, and this implies
that $C_\alpha=C\rtimes p_\alpha$ is isomorphic to $C_e=C\rtimes p_e$ as left
$C_e$-comodules. In order to show that $\ul{C}$ is strong, it suffices to show that
$\delta_{\alpha,\alpha^{-1}}$ is monic, for all $\alpha\in G$, see \prref{1.3}
For all $c\in C$, we have
\begin{eqnarray*}
&&\hspace*{-2cm}
(g_{\alpha,\alpha^{-1}}\ot \lambda_{\alpha^{-1}})(\delta_{\alpha,\alpha^{-1}}(c))=
g_{\alpha,\alpha^{-1}}(c_{(1)})( \lambda_{\alpha^{-1}}\circ \lambda_\alpha)(c_{(2)})
f_{\alpha,\alpha^{-1}}(c_{(3)})\\
&\equal{\equref{TC}}&g_{\alpha,\alpha^{-1}}(c_{(1)})f_{\alpha,\alpha^{-1}}(c_{(2)})c_{(3)}=c,
\end{eqnarray*}
where we used the fact that $\lambda_e=C$ ($f$ is normalized).\\
$\ul {(4)\Rightarrow (1)}$. From \thref{1.7}, we know that the corestriction
$\Delta_{\alpha,\beta}:\ C_{\alpha\beta}\to C_\alpha\Box_{C_e} C_\beta$ has an inverse
$\nabla_{\alpha,\beta}$. Let $\phi_\alpha:\ C_\alpha\to C_e$ be a left $C_e$-colinear
isomorphism, for all $\alpha\in G$.
Then we have the following formulas:
\begin{eqnarray}
\Delta_{e,e}(\phi_\alpha(c))&=& c_{(1,e)}\ot \phi_\alpha(c_{(2,\alpha)});\eqlabel{3.4.3}\\
\Delta_{e,\alpha}(\phi^{-1}_\alpha(d))&=&
d_{(1,e)}\ot \phi_\alpha^{-1}(d_{(2,e)})\eqlabel{3.4.4},
\end{eqnarray}
for all $c\in C_\alpha$ and $d\in C_e$. For $c\in C_{\alpha^{-1}}$, we have that
$c_{(1,\alpha^{-1})}\ot \phi_{\alpha}^{-1}(c_{(2,e)})\in C_{\alpha^{-1}}\Box_{C_e}C_\alpha$.
Indeed,
\begin{eqnarray*}
&&\hspace*{-2cm}
c_{(1,\alpha^{-1})}\ot \Delta_{e,\alpha}(\phi_{\alpha}^{-1}(c_{(2,e)}))
\equal{\equref{3.4.4}} c_{(1,\alpha^{-1})}\ot c_{(2,e)}\ot \phi_\alpha^{-1}(c_{(3,e)})\\
&=& \Delta_{\alpha^{-1},e}(c_{(1,\alpha^{-1})})\ot \phi_\alpha^{-1}(c_{(2,e)}).
\end{eqnarray*}
This implies that we have a well-defined map
$$v_\alpha=\varepsilon\circ \nabla_{\alpha^{-1},\alpha}\circ (C_{\alpha^{-1}}\ot\phi_\alpha^{-1})
\circ \Delta_{\alpha^{-1},e}: C_{\alpha^{-1}}\to k.$$
We also consider
$u_\alpha=\varepsilon\circ \phi_\alpha:\ C_\alpha\to k$.
For every $c\in C_e$, we now have that
\begin{eqnarray*}
&&\hspace*{-15mm}
u_\alpha(c_{(1,\alpha)})v_\alpha(c_{(2,\alpha^{-1})})\\
&=&\varepsilon(\phi_\alpha(c_{(1,\alpha)}))(\varepsilon\circ \nabla_{\alpha^{-1},\alpha})
(c_{(2,\alpha^{-1})}\ot \phi_\alpha^{-1}(c_{(3,e)}))\\
&\equal{\equref{3.4.4}}&
\varepsilon(\phi_\alpha(\phi_\alpha^{-1}(c)_{(1,\alpha)}))
(\varepsilon\circ \nabla_{\alpha^{-1},\alpha})
(\phi_\alpha^{-1}(c)_{(2,\alpha^{-1})}\ot \phi_\alpha^{-1}(c)_{(3,\alpha)})\\
&=&\varepsilon(\phi_\alpha(\phi_\alpha^{-1}(c)_{(1,\alpha)}))
(\varepsilon\circ \nabla_{\alpha^{-1},\alpha}\circ \Delta_{\alpha^{-1},\alpha})
(\phi_\alpha^{-1}(c)_{(2,e)})\\
&=& \varepsilon(\phi_\alpha(\phi_\alpha^{-1}(c)_{(1,\alpha)}))
\varepsilon(\phi_\alpha^{-1}(c)_{(2,e)})=\varepsilon(\phi_\alpha(\phi_\alpha^{-1}(c)))
=\varepsilon(c);\\
&&\hspace*{-15mm}
v_\alpha(c_{(1,\alpha^{-1})})u_\alpha(c_{(2,\alpha)})\\
&=&
(\varepsilon\circ\nabla_{\alpha^{-1},\alpha})(c_{(1,\alpha^{-1})} \ot \phi_\alpha^{-1}(c_{(2,e)}))
(\varepsilon\circ \phi_\alpha)(c_{(3,\alpha)})\\
&\equal{\equref{3.4.3}}&
(\varepsilon\circ\nabla_{\alpha^{-1},\alpha})(c_{(1,\alpha^{-1})}
\ot\phi_\alpha^{-1}(\phi_\alpha(c_{(2,\alpha)})_{(1,e)})
\varepsilon(\phi_\alpha(c_{(2,\alpha)})_{(2,e)}))\\
&=&
(\varepsilon\circ\nabla_{\alpha^{-1},\alpha})(c_{(1,\alpha^{-1})}
\ot\phi_\alpha^{-1}(\phi_\alpha(c_{(2,\alpha)})))\\
&=&
(\varepsilon\circ\nabla_{\alpha^{-1},\alpha}\circ\Delta_{\alpha^{-1},\alpha})(c)=
\varepsilon(c),
\end{eqnarray*}
as needed.
\end{proof}

If there exists a morphism of $G$-coalgebras $\ul{u}:\ \ul{C}\to k\lan G\ran$, then $\ul{C}$ is cocleft.
We have the following characterization of this situation.

\begin{theorem}\thlabel{3.5}
For a $G$-coalgebra $\ul{C}$, the following assertions are equivalent:
\begin{enumerate}
\item $\ul{C}$ is isomorphic to a smash coproduct $G$-coalgebra;
\item there exists a morphism of $G$-coalgebras $\ul{u}:\ \ul{C}\to k\lan G\ran$.
\end{enumerate}
\end{theorem}

\begin{proof}
$\ul {(1)\Rightarrow (2)}$. Let $\ul{C}=C\rtimes k\lan G\ran$ be a smash coproduct
$G$-coalgebra. Then $\delta_{\alpha,\beta}(c)=c_{(1)}\ot \lambda_\alpha(c_{(2)})$,
for all $c\in C$. The map $\ul{u}:\ C\rtimes k\lan G\ran\to k\lan G\ran$, 
$u_\alpha(c\rtimes p_\alpha)=\varepsilon_C(c)p_\alpha$,
is a morphism of $G$-coalgebras since
\begin{eqnarray*}
&&\hspace*{-2cm}
((u_\alpha\ot u_\beta)\circ\Delta_{\alpha,\beta})(c\rtimes p_{\alpha\beta})
= \varepsilon_C(c_{(1)})p_\alpha \ot \varepsilon_C(\lambda_\alpha(c_{(2)}))p_\beta\\
&=&\varepsilon_C(c)p_{\alpha}\ot p_{\beta}=(\Delta_{\alpha,\beta}
\circ u_{\alpha\beta}(c\rtimes p_{\alpha\beta});\\
&&\hspace*{-2cm}
\varepsilon(c\rtimes p_e)=\varepsilon_C(c)=\varepsilon(u_e(c\rtimes p_e)).
\end{eqnarray*}
$\ul {(2)\Rightarrow (1)}$. As we have already mentioned, $\ul{C}$ is cocleft. 
The convolution inverse of $u_\alpha$ is just $v_\alpha=u_{\alpha^{-1}}$. In 
the proof of $\ul {(1)\Rightarrow (2)}$ in \thref{3.4}, we have seen that
$\ul{C}\cong C_e\rtimes_f k\lan G\ran$, with
$$f_{\alpha,\beta}(c)=u_\alpha(c_{(1,\alpha)})u_\beta(c_{(2,\beta)})
u_{\beta^{-1}\alpha^{-1}}(c_{(3,\beta^{-1}\alpha^{-1})})=u_e(c)=\varepsilon(c).$$
\end{proof}

In \thref{3.5}, we have characterized crossed coproduct $G$-coalgebras with trivial
factor set. We can also characterize when the maps $\lambda_\alpha$ are the
identity maps.

\begin{theorem}\thlabel{3.6}
For a $G$-coalgebra $\ul{C}$, the following conditions are equivalent.
\begin{enumerate}
\item There exists a convolution invertible $\ul{u}:\ \ul{C}\to k\lan G\ran$ such that
\begin{equation}\eqlabel{3.6.1}
u_\alpha(c_{(1,\alpha)})c_{(2,e)}=c_{(1,e)}u_\alpha(c_{(2,\alpha)}),
\end{equation}
for all $\alpha\in G$ and $c\in C_\alpha$;
\item $\ul{C}$ is isomorphic to a crossed coproduct $G$-coalgebra $C\rtimes_f k\lan G\ran$,
with $\lambda_\alpha=C$, for all $\alpha\in G$;
\item $\ul{C}$ is isomorphic to a crossed coproduct $G$-coalgebra $C\rtimes_f k\lan G\ran$,
with $f$ normalized and with $\lambda_\alpha=C$, for all $\alpha\in G$;
\item $\ul{C}$ is a strong $G$-coalgebra, and every $C_\alpha$ is isomorphic to $C_e$
as a $C_e$-bicomodule.
\end{enumerate}
\end{theorem}

\begin{proof}
$\ul {(1)\Rightarrow (2)}$. It follows from \equref{3.6.1} that the maps $\lambda_\alpha$
constructed in the proof of $\ul {(1)\Rightarrow (2)}$ in \thref{3.4} are
equal to the identity map on $C_e$.\\
$\ul {(2)\Rightarrow (3)}$. If the maps $\lambda_\alpha$ are equal to $C$ in the proof of
\prref{3.2}, then we also have that $\lambda'_\alpha=C$, for all $\alpha\in G$.\\
$\ul {(3)\Rightarrow (4)}$. In the proof of $\ul {(3)\Rightarrow (4)}$ in \thref{3.4}, it is shown
that $\ul{C}$ is strong, and left $C_e$-colinear maps $C_\alpha\to C_e$ are given;
using the fact that $\lambda_\alpha=C$, we can easily show that these maps are also
right $C_e$-colinear.\\
$\ul {(4)\Rightarrow (1)}$. We are in the situation of the proof of $\ul {(4)\Rightarrow (1)}$
in \thref{3.4}, with the additional hypothesis that $\phi_\alpha$ is right $C_e$-colinear,
that is,
$$\Delta_{e,e}(\phi_\alpha(c))=\phi_\alpha(c_{(1,\alpha)})\ot c_{(2,e)}.$$
Applying $\varepsilon$ to the first tensor factor, we find that $\phi_\alpha(c)=
u_\alpha(c_{(1,\alpha)})c_{(2,e)}$. If we apply $\varepsilon$ to the second tensor factor
of \equref{3.4.3}, then we obtain that $\phi_\alpha(c)=c_{(1,e)}u_\alpha(c_{(2,\alpha)})$,
and \equref{3.6.1} follows.
\end{proof}

\section{Cohomology}\selabel{4}
Let $C$ be a cocommutative coalgebra, $\lambda$ a collection of coalgebra
endomorphisms of $C$ and $f$ a factor set (see \prref{3.1}). Then \equref{3.1.1}
is equivalent to $\lambda_e=C$, and \equref{TC} is equivalent to
$\lambda_{\alpha\beta}=\lambda_\beta\circ\lambda_\alpha$. This means that
$C$ is a right $G$-module coalgebra, with right $G$-action $c\cdot \alpha=
\lambda_\alpha(c)$. Then $C^*$ is a left $G$-module algebra, with left $G$-action
$\lan \alpha\cdot f,c\ran=\lan f,c\cdot \alpha\ran$.
The abelian group $\GG_m(C^*)$ consisting of convolution invertible elements 
of $C^*$ is a left $G$-module, and we can consider the cohomology groups
$H^n(G,\GG_m(C^*))$. Condition \equref{C} can be rewritten as
$(\alpha\cdot f_{\beta,\gamma})*f_{\alpha,\beta\gamma}=f_{\alpha,\beta}*
f_{\alpha\beta,\gamma}$, 
which means precisely that $f\in Z^2(G,\GG_m(C^*))$ is a 2-cocycle. \equref{CU}
follows from \equref{C}, after we 
successively take $\alpha=\beta=e$ and $\beta=\gamma=e$ in \equref{C}.

\begin{lemma}\lelabel{4.1}
Let $C$ be a cocommutative right $G$-module coalgebra, and let $f\in Z^2(G,\GG_m(C^*))$. 
If $f'$ is the
cocycle obtained from $f$ using the construction in \prref{3.2}, then $f^{-1}*f'\in B^2(G,\GG_m(C^*))$.
\end{lemma}

\begin{proof}
For every $\alpha\in G$, we define $h_\alpha\in C^*$ as follows $h_\alpha=\varepsilon_C$
if $\alpha\neq e$ and $h_e=f_{e,e}$. We can easily compute
$\delta_1(h)_{\alpha,\beta}=h_\alpha*h^{-1}_{\alpha\beta}*(\alpha\cdot h_\beta)$.
We find that $\delta_1(h)_{e,\alpha}=f_{e,\alpha}$ and $\delta_1(h)_{\alpha,e}=
f_{\alpha,e}$, for all $\alpha\in G$, $\delta_1(h)_{\alpha,\alpha^{-1}}=g_{e,e}$,
for all $\alpha\neq e$ and $\delta_1(h)_{\alpha,\beta}=\varepsilon_C$, if $\alpha\neq
e$, $\beta\neq e$ and $\alpha\beta\neq e$. Looking at the explicit formula of $f'$ in
the proof of \prref{3.2}, we see that $f_{\alpha,\beta}=f'_{\alpha,\beta}*
\delta_1(h)_{\alpha,\beta}$.
\end{proof}

\begin{proposition}\prlabel{4.2}
Let $C$ be a cocommutative right $G$-module coalgebra, and take $f,f'\in Z^2(G,\GG_m(C^*))$.
The crossed coproduct $G$-coalgebra $C\rtimes_{f} k\lan G\ran$ and
$C\rtimes_{f'} k\lan G\ran$ are isomorphic if and only if there exists a coalgebra automorphism
$\varphi$ of $C$ such that $[f]=[f'\circ\varphi]$ in $H^2(G,\GG_m(C^*))$.
\end{proposition}

\begin{proof}
It follows from \prref{3.2} and \leref{4.1} that we may restrict attention to the situation where
$f$ and $f'$ are normalized cocycles. First assume that $\{\phi_\alpha~|~
\alpha\in G\}$ is a family of $k$-linear maps $C\to C$ such that
$(\phi_\alpha\rtimes p_\alpha)_{\alpha\in G}$ is an isomorphism between the
$G$-coalgebras $C\rtimes_{f} k\lan G\ran$ and $C\rtimes_{f'} k\lan G\ran$.
Then we have for al $\alpha,\beta\in G$ and $c\in C$ that
\begin{eqnarray}
\phi_\alpha(c_{(1)})\ot \phi_\beta(c_{(2)}\cdot\alpha)f_{\alpha,\beta}(c_{(3)})\nonumber\\
&=&
\phi_{\alpha\beta}(c)_{(1)}\ot (\phi_{\alpha\beta}(c)_{(2)}\cdot \alpha)
f'_{\alpha,\beta}(\phi_{\alpha\beta}(c)_{(3)}),\eqlabel{4.2.1}
\end{eqnarray}
and
$
\varepsilon_C=\varepsilon_C\circ \phi_e$,
where we used the fact that $f$ and $f'$ are normalized, so that
$\varepsilon_C=g_{e,e}=g'_{e,e}$. In particular, $\phi_e$ is a coalgebra automorphism of $C$.
Consider the maps $h_\alpha=\varepsilon_C\circ \phi_\alpha\in C^*$.
Take $\alpha=e$ in \equref{4.2.1} and apply
$\varepsilon_C$ to the second tensor factor. Again using the normality of $f$ and $f'$,
we find that
\begin{equation}\eqlabel{4.2.3}
\phi_e(c_{(1)})h_\beta(c_{(2)})=\phi_\beta(c).
\end{equation}
It is then easy to show that $h_\alpha^{-1}=\varepsilon_C\circ\phi_\alpha^{-1}\circ \phi_e$
is the convolution inverse of $h_\alpha$. Now apply $\varepsilon_C\ot\varepsilon_C$ to
\equref{4.2.1}, to obtain that
$$h_\alpha(c_{(1)})(\alpha\cdot h_\beta)(c_{(2)})f_{\alpha,\beta}(c_{(3)})=
f'_{\alpha,\beta}(\phi_{\alpha\beta}(c))\equal{\equref{4.2.3}}f'_{\alpha,\beta}(\phi_e(c_{(1)})
h_{\alpha\beta}(c_{(2)})),$$
or 
\begin{equation}\eqlabel{4.2.4}
f_{\alpha,\beta}*h_\alpha*(\alpha\cdot h_\beta)*h^{-1}_{\alpha\beta}=f'_{\alpha,\beta}
\circ\phi_e,
\end{equation}
Conversely, assume that $[f]=[f'\circ\varphi]$, that is, there exists a family of convolution
invertible maps $\{h_\alpha~|~\alpha\in G\}\subset C^*$ such that \equref{4.2.4} holds (with $\phi_e=\varphi$). Then we define $\phi_\alpha$ by \equref{4.2.3}. Straightforward computations show that the
$\phi_\alpha$ define an isomorphism of $G$-coalgebras.
\end{proof}

Now let $\ul{C}$ be a cocleft $G$-coalgebra, and assume that $C_e$ is cocommutative.
Let $u_\alpha:\ C_\alpha\to k$ and
$v_\alpha:\ C_{\alpha^{-1}}\to k$ be $k$-linear maps satisfying \equref{3.2.2}.
Recall from the proof of $\ul{(1)\Rightarrow (2)}$ in \thref{3.4} that we have
maps $\lambda_\alpha:\ C_e\to C_e$ and $f_{\alpha,\beta}:\ C_e\to k$ defined by
\equref{3.4.1} and \equref{3.4.2}, and that $\ul{C}\cong C\rtimes_f k\lan G\ran$.
We have seen at the beginning of this Section that $C_e$ is a right $G$-module coalgebra,
with right $G$-action 
$c\cdot \alpha=\lambda_\alpha(c)$.

\begin{lemma}\lelabel{4.3}
With notation as above, we have:
\begin{enumerate}
\item the maps $\lambda_\alpha$ are independent of the choice of $\ul{u}$;
\item for all $\alpha\in G$ and $c\in C_\alpha$, we have
$c_{(1,\alpha)}\ot c_{(2,e)}=c_{(2,\alpha)}\ot c_{(1,e)}\cdot \alpha$.
\end{enumerate}
\end{lemma}

\begin{proof}
Applying comultiplication maps to the first tensor factor of the cocommutativity relation
$
c_{(1,e)}\ot c_{(2,e)}= c_{(2,e)}\ot c_{(1,e)}$,
we find the following relation
\begin{equation}
c_{(1,\alpha^{-1})}\ot c_{(2,\alpha)}\ot c_{(3,e)} 
=
c_{(2,\alpha^{-1})}\ot c_{(3,\alpha)} \ot c_{(1,e)}.
\eqlabel{4.3.2}
\end{equation}
Let $\ul{u}':\ \ul{C}\to k\lan G\ran$ be another convolution invertible morphism, with
corresponding action $\bullet$. Then
\begin{eqnarray*}
c\cdot \alpha
&=& u_\alpha(c_{(1,\alpha)})c_{(2,e)}v_\alpha(c_{(3,\alpha^{-1})})
u'_\alpha(c_{(4,\alpha)})v'_\alpha(c_{(5,\alpha^{-1})})\\
&\equal{\equref{4.3.2}}& u_\alpha(c_{(1,\alpha)})v_\alpha(c_{(2,\alpha^{-1})})
u'_\alpha(c_{(3,\alpha)})c_{(4,e)}v'_\alpha(c_{(5,\alpha^{-1})})\\
&=& u'_\alpha(c_{(1,\alpha)})c_{(2,e)}v'_\alpha(c_{(3,\alpha^{-1})})=c\bullet\alpha.
\end{eqnarray*}
and this proves (1). (2) is shown as follows.
\begin{eqnarray*}
&&\hspace*{-2cm}
c_{(2,\alpha)}\ot c_{(1,e)}\cdot \alpha
= c_{(4,\alpha)}\ot u_\alpha(c_{(1,\alpha)})c_{(2,e)}v_\alpha(c_{(3,\alpha^{-1})})\\
&\equal{\equref{4.3.2}}& c_{(3,\alpha)}\ot u_\alpha(c_{(1,\alpha)})c_{(4,e)}v_\alpha(c_{(2,\alpha^{-1})})
= c_{(1,\alpha)}\ot c_{(2,e)}.
\end{eqnarray*}
\end{proof}

Let
$\Omega_{\ul{C}}=\{\ul{u}:\ \ul{C}\to k\lan G\ran~|~\ul{c}~{\rm is~a~morphism~of~}G~{\rm
coalgebras}\}$.
$\ul{C}$ is isomorphic to a smash coproduct $G$-coalgebra if and only if $\Omega_{\ul{C}}
\neq \emptyset$. If we know one element of $\Omega_{\ul{C}}$, then we can describe all
the others using cohomology.

\begin{proposition}\prlabel{4.3}
If $\Omega_{\ul{C}}\neq\emptyset$, then there is a bijection
$$\phi:\ \Omega_{\ul{C}}\to Z^1(G,\units(C_e^*)).$$
\end{proposition}

\begin{proof}
Fix $\ul{u}^0\in \Omega_{\ul{C}}$, For $\ul{u}\in \Omega_{\ul{C}}$, we define
$\phi(\ul{u})=\theta:\ G\to \GG_m(C_e^*)$ by the formula
$$\theta_\alpha(c)=u_\alpha(c_{(1,\alpha)})u^0_{\alpha^{-1}}(c_{(2,\alpha^{-1})}).$$
We show that $\theta_\alpha$ is a 1-cocycle. For all $c\in C_e$, we have
\begin{eqnarray*}
&&\hspace*{-10mm}
((\alpha\cdot \theta_\beta)*\theta_\alpha)(c)
=\theta_\beta(c_{(1,e)}\cdot\alpha)\theta_\alpha(c_{(2,e)})\\
&=&\theta_\beta\bigl(u_\alpha(c_{(1,\alpha)})c_{(2,e)}u_{\alpha^{-1}}(c_{(3,\alpha^{-1})}\bigr)
\theta_\alpha(c_{(4,e)})\\
&=&
u_\alpha(c_{(1,\alpha)}) u_\beta(c_{(2,\beta)})u^0_{\beta^{-1}}(c_{(3,\beta^{-1})}
u_{\alpha^{-1}}(c_{(4,\alpha^{-1})})u_\alpha(c_{(5,\alpha)})u^0_{\alpha^{-1}}(c_{(6,\alpha^{-1})})\\
&=&
u_\alpha(c_{(1,\alpha)}) u_\beta(c_{(2,\beta)})u^0_{\beta^{-1}}(c_{(3,\beta^{-1})}
u^0_{\alpha^{-1}}(c_{(4,\alpha^{-1})})\\
&=& u_{\alpha\beta}(c_{(1,\alpha\beta)})u^0_{\beta^{-1}\alpha^{-1}}(c_{(2,
\beta^{-1}\alpha^{-1})})=\theta_{\alpha\beta}(c).
\end{eqnarray*}
It follows that $\alpha\cdot\theta_\beta*\theta_\alpha=\theta_{\alpha\beta}$, which is
precisely the cocycle relation. $\phi^{-1}$ is given by the formula
$\phi^{-1}(\theta)=\ul{u}$, with
$u_\alpha(c)=\theta_\alpha(c_{(1,e)})u^0_\alpha(c_{(2,\alpha)})$.t
\end{proof}

On $\Omega_{\ul{C}}$, we define the following equivalence relation:
$\ul{u}\sim \ul{u}'$ if and only if there exists a convolution invertible $f:\ C_e\to k$
such that
$$f(c_{(1,e)})u_\alpha(c_{(2,\alpha)})=u'_\alpha(c_{(1,e)})f(c_{(2,\alpha)}),$$
for all $\alpha\in G$ and $c\in C_\alpha$. We denote
$$\ul{\Omega}_{\ul {C}}= {\Omega}_{\ul {C}}/\sim.$$

\begin{proposition}\prlabel{4.4}
If $\Omega_{\ul{C}}\neq\emptyset$, then there is a bijection
$$\ul{\phi}:\ \ul{\Omega}_{\ul{C}}\to H^1(G,\units(C_e^*)).$$
\end{proposition}

\begin{proof}
$\ul{u}\sim \ul{u}^0$ if and only if there exists $f\in C_e^*$ with convolution inverse $g$
such that
$$u_\alpha^0(c)=f(c_{(1,e)})u_\alpha(c_{(2,\alpha)})g(c_{(3,e)}),$$
for all $\alpha\in G$, $c\in C_\alpha$. This is equivalent to
\begin{eqnarray*}
\theta_\alpha(c)&=&
u_\alpha(c_{(1,\alpha)}) f(c_{(2,e)})u_{\alpha^{-1}}(c_{(3,\alpha^{-1})})g(c_{(4,e)})\\
&=&(a\cdot f)(c_{(1,e)})g(c_{(2,e)})= ((a\cdot f)*g)(c).
\end{eqnarray*}
This is equivalent to the existence of $f\in C_e^*$ with convolution inverse $g$
such that $\theta_\alpha=(\alpha\cdot f)*g$, which means precisely that $\theta_\alpha
\in B^1(G,C_e^*)$.
\end{proof}

\end{document}